\documentclass[11pt]{amsart}
\usepackage{amsmath}
\usepackage{amssymb}
\usepackage{amsfonts}
\theoremstyle{plain}
\newtheorem{theorem}{Theorem}[section]

\newtheorem{corollary}[theorem]{Corollary}

\newtheorem{lemma}[theorem]{Lemma}

\newtheorem{proposition}[theorem]{Proposition}
\newtheorem{remark}[theorem]{Remark}
\numberwithin{equation}{section}
\newtheorem*{hyp*}{Hypothesis} 
\newtheorem{hyp}[theorem]{Hypothesis} 

\begin{document}

\title[$2\times2$ Hypergeometric operators with diagonal eigenvalues]
{$2\times2$ Hypergeometric operators with diagonal eigenvalues}
\author[C. Calder\'on]{C. Calder\'on}
\author[Y. Gonz\'alez]{Y. Gonz\'alez}
\author[I. Pacharoni]{I. Pacharoni}
\author[S. Simondi]{S. Simondi}
\author[I. Zurri\'an]{I. Zurri\'an}
% \address[A. U. Thor]{Author address line 1\\
% Author address line 2}
% \curraddr[A. U. Thor]{Author current address line 1\\
% Author current address line 2}
% \email[A. U. Thor]{author@institute.edu}
\email[C. Calder\'on]{celeste.calderon@fce.uncu.edu.ar}
\email[Y. Gonz\'alez]{ygonzalez@fcen.uncu.edu.ar}
\email[I. Pacharoni]{pacharon@famaf.unc.edu.ar}
\email[S. Simondi]{ssimondi@uncu.edu.ar}
\email[I. Zurri\'an]{zurrian@famaf.unc.edu.ar}
% \urladdr{http://www.author.institute}
\thanks{This research was supported in part by CONICET grant PIP 112-200801-01533, SeCyT-UNC grant 05/B368 and SECTYP-UNCUYO grant M052}
\address[C. Calder\'on, Y. Gonz\'alez, S. Simondi]{Facultad de
Ciencias Exactas y Naturales, Universidad Nacional de Cuyo, 5500 Mendoza,
Argentina}
\address[I. Pacharoni, I. Zurri\'an]{FAMAF-CIEM, Ciudad Universitaria, 5000 C\'ordoba, Argentina}
% \email[A. U. Thor]{author@institute.edu}
% \urladdr{http://www.author.institute}
% \thanks{The Author thanks V. Exalted}


\begin{abstract}
In this work we classify all the order-two Hypergeometric operators $D$, symmetric with respect to some $2\times 2$ irreducible matrix-weight $W$ such that $DP_n=P_n\left(\begin{smallmatrix}
\lambda_n&0\\0&\mu_n
\end{smallmatrix}
\right)$ with no repetition among the eigenvalues $\{\lambda_n,\mu_n\}_{n\in\mathbb N_0}$, where 
$\{P_n\}_{n\in\mathbb N_0}$ is the (unique) sequence of monic orthogonal polynomials
with respect to $W$. 

We obtain, in a very explicit way,  a three parameter  family of such operators and weights. We also give the corresponding monic orthongonal polynomials, their three term recurrence relation and their squared matrix-norms.
\end{abstract}
\maketitle
%\tableofcontents
\section{Introduction. }

M.G. Krein started the study of matrix-valued orthogonal polynomials in 1949, settling the general and basic theory in \cite{K49,K71}. The situations in which these polynomials enjoy some extra property, such as the one singled out in \cite{DG86} and generally known as {\it the bispectral property}, may be very interesting and promising in many areas of mathematics and its applications.

The search for concrete instances enjoying bispectrality has received a certain amount of
 attention, after the work started by A. Dur\'an in \cite{D97}. The collection of known examples has been growing in the last fifteen years after the discovery of irreducible examples in \cite{GPT01,GPT02a,GPT03,G03} and in \cite{DG04}.
 
We are interested in $2\times2$ matrix weights  $W$ that admit a symmetric differential operator $D$ of hypergeometric type in the sense of \cite{T03}.

It is known that if $\{P_n\}_{n\in\mathbb N_0}$ is the (unique) sequence of monic orthogonal polynomials with respect to $W$ then we have 
$$
DP_n=P_n\Lambda_n
$$
for every $n$, where $\Lambda_n$ is a $2\times2$ matrix. 

Among the applications of matrix-valued bispectral families one can find the study of time and band limiting  over a non-commutative ring and matrix-valued commuting operators,
 see \cite{GPZ15,CG15,GPZ17,CG17,CGPZ17,GPZ18}. In the scalar version of these applications, the condition that the operator $D$ should have ``simple spectrum'' is crucial. In the matrix-valued case the analogous condition is not completely clear yet and this is matter of further study. Nevertheless, it is necessary to ask for commutativity of all the eigenvalues $ \{\Lambda_n\}_{n\in\mathbb N_0}$ and that all the scalar eigenvalues of all these matrix eigenvalues are different to each other, this is one of the reasons why in this work we  will restrict our attention to the matrix-valued hypergeometric operators such that the corresponding eigenvalues fulfill these conditions, i.e., every $\Lambda_n$ is of the form $\Lambda_{n}=
\left(\begin{smallmatrix}
\lambda _{n} & 0 \\
0 & \mu _{n}
\end{smallmatrix}\right)
,$
with no repetition among the $\{\lambda_n,\mu_n\}_{n\in\mathbb N_0}$. 

\

We classify explicitly all  order-two Hypergeometric operators $D$, symmetric with respect to some $2\times 2$ irreducible matrix-weight $W$ such that 
the matrix eigenvalues $\Lambda_n$ are diagonal matrices  with no repetitions in their entries.
We explicitly give the expression of every member of this family which depends on three parameters $\alpha, \beta ,v\in\mathbb R$. Also, we provide the explicit expression of $W$ as well as of the matrices involved in the three-term recurrence relation satisfied by the monic orthogonal polynomials (see Theorem \ref{CUVAnBn}) and their squared matrix norms (see \eqref{AB}). The three-parameter family we exhibit here generalizes previous examples such as \cite{G03} and \cite{PZ16}.

\smallskip

In Section \ref{prel} we recall some general results and properties regarding the three term recurrence relation satisfied by matrix-valued monic orthogonal polynomials and the notion of equivalence among matrix weights.

In Section \ref{nec},  we find explicitly the necessary conditions 
%for the expression of all the operators $D$ 
such that there exists a weight $W$ for which $D$ is symmetric and the eigenvalues $\Lambda_n$ are diagonal matrices with no repetition in their entries.

In Section \ref{Existence}, we give explicitly all the operators $D$ and the weights $W$ with the desired properties. The main result is given in Theorem \ref{main}.  
 Finally, we make some comments connecting the results given here with those in \cite{T03} and with the family of matrix  orthogonal polynomials obtained in \cite{PZ16}.

It is worth to mention that most of results in this paper are simple  but many of the proofs require  technical and long computations. 
In some proofs, we do not fully develop these calculations, nevertheless, we give the main lines and the general scheme for the reader interested in reproducing them.

For previous works on matrix examples of Jacobi type the reader may consider \cite{G03,GPT03,PZ16} for size $2$ and \cite{GPT05,PT06,PR08,KPR12,KPR13,PTZ14,KRR17} for arbitrary size; the methods in those works are completely different to the one employed in this paper. Regarding the matrix-Bochner problem and the studies of the algebra $\mathcal D(W)$ the reader may see the very  recent works \cite{C18,CY18} for general weights and sizes or the older works \cite{T11,Z16} for $2\times2$ cases of Hermite and Gegenbauer, respectively.

\section{Preliminaries}\label{prel}

In this section we introduce some basic notions and results. 
Given a self-adjoint $N\times N$ positive matrix-valued weight function $W(t)$ defined in $(0,1)$, we consider
the skew symmetric bilinear form defined by the  matrix
\[
\left\langle P,Q\right\rangle _{W}=\int_0^1 P^{\ast }(t)W\left( t\right) Q\left(
t\right) dt,
\]%
 for any pair of  $N\times N$ matrix-valued
functions $P(t)$ and $Q(t)$, 
where $P^{\ast }(t)$ denote the conjugate transpose of $P(t)$.

 By following a standard argument,
given for instance in \cite{K49} or \cite{K71},
one shows that the monic orthogonal polynomials
satisfy a three term recurrence relation
\begin{equation}\label{ttrr}
tP_{n}=P_{n+1}+P_{n}B_{n}+P_{n-1}A_{n},\quad \text{ for }  n\in\mathbb N_0, 
\end{equation}
with the convention $P_{-1}=0$ and $A_0=I$, where $A_n, B_n$ are  $N\times N$ matrices depending on $n$ and not in $t$, with $A_n$ nonsingular for any $n$.
Moreover, for the sequence $\{S_n\}_{n\in\mathbb{N}}$, given by the square norms of the monic orthogonal polynomials, i.e. $S_n=\|P_n\|^2=\left\langle P,Q\right\rangle _{W}$, we have 
that the following conditions are fulfilled,
\begin{equation}\label{AB}
A_{n+1}=S_{n}^{-1}S_{n+1,}\qquad S_{n}B_{n}\text{ is hermitian},\quad \forall n\in\mathbb N_0,
\end{equation}
see \cite[Theorem 6.1]{DLR96} or \cite[p.96]{CG05} and the references given therein.

\smallskip

Two weights $W$ and $\tilde W$ are said to be \emph{similar} if there exists
a nonsingular matrix $M$, which does not depend on $t$, such that
\[
\tilde W(t)=M^* W(t)M, \quad \text{ for all } t\in (0,1).
\]
A matrix-weight  $W$ {\it reduces} to a smaller size if there exists a nonsingular
matrix $M$ such that
\[
M^* W(t) M=
\begin{pmatrix}
W_1(t) & 0 \\
0 & W_2(t)%
\end{pmatrix}%
, \quad \text{ for all } t\in (0,1),
\]
where $W_1$ and $W_2$ are weights of smaller size. A matrix-weight  $W$ is said to be {\it irreducible} if it does not reduce to a smaller size. From {\cite[Theorem 4.5]{TZ16}}, we have the following criterium of reducibility for our particular case.
\begin{theorem}\label {red}
The weight $W$ reduces if and only if  there is a nonscalar matrix commuting with every $A_n$ and $B_n$ for $n\in\mathbb N_0$.
\end{theorem}

\smallskip 
We say that a differential operator $D$ is \emph{symmetric} if $%
\langle DP,Q\rangle=\langle P,DQ\rangle$, for all $P,Q$ matrix-valued
polynomials. 

\medskip

\begin{remark}
Notice that if $\{P_n\}_{n\in\mathbb N_0}$ is the sequence of monic
orthogonal polynomials with respect to $W$, and $M$ is a nonsingular matrix,
then $\{M^{-1} P_n M\}_{n\in\mathbb N_0}$ is the sequence of monic
orthogonal polynomials with respect to $\tilde W=M^*W M$. In the same way, a differential operator $D$ is symmetric with respect to a weight $W$ if and
only if $\tilde D= M^{-1}DM$ is symmetric with respect to $\tilde W=M^{*}W M$.
\end{remark}

\medskip
\section{Necessary Conditions for $D$}\label{nec}

We are interested in finding all the pairs $(W,D)$ such that $W(t) $ is a $2\times2$-weight function defined on  $\left( 0,1\right) $ and 
\begin{equation}
D=
%D(C,U,V)\doteq 
t(1-t)\frac{d^{2}}{dt^{2}}+\left( C-tU\right) \frac{d}{dt}-V, \quad \text{ with } U,V,C\in
 \mathbb{R}
^{2\times 2}  \label{OD},
\end{equation}%
is a $2\times2$ matrix-valued hypergeometric differential operator symmetric with respect to ${W}$.

For each such a weight $W$, there is a unique sequence of monic orthogonal polinomials 
$\{P_n\}_{n\in \mathbb N_0}$ in $\operatorname{Mat}_2(\mathbb{C})$.  
It is known (see \cite{GT07}) that if $\{P_n\}_{n\in\mathbb N_0}$ is the unique sequence of 
monic orthogonal polinomials, then every polynomial $P_n$ is an eigenfunction of  such an operator, i.e.,
$$DP_n=P_n \Lambda_n, \quad \forall n\in\mathbb N_0,$$ 
where $\Lambda_n$ is a $2\times2$-matrix for any $n\in\mathbb N_0$.

\

Our aim is to classify all the operators $D=D({C,U,V})$ of the form \eqref{OD} that 
are symmetric with respect to any $2\times2$-matrix weight $W$ and the corresponding eigenvalues   of the monic orthogonal polynomials  are diagonal matrices, 
\[
\Lambda_{n}=
\begin{pmatrix}
\lambda _{n} & 0 \\
0 & \mu _{n}
\end{pmatrix}
,\]
with no repetition among $\{\lambda_n,\mu_n\}_{n\in\mathbb N_0}$ (this ``no repetition'' means that $\lambda_j\ne\mu_k$ for any $j,k\in\mathbb N_0$ and that
$\lambda_j\ne\lambda_k$ and $\mu_j\ne\mu_k$ when  $j\ne k$). 

\medskip
Of course we are not paying attention to the operators for which $W$ reduces to scalar weights, since these are given by the well known classical Jacobi differential operators. In particular, we are not considering the operators $D$ for which $C,U,V$ are all diagonal matrices.

Having said this, we are assuming for the rest of this section the following hypothesis.

\begin{hyp}\label{h} Let $D$ denote the differential operator given by \eqref{OD}, symmetric with respect to a $2\times2$-irreducible weight $W$ defined on $(0,1)$ and $\{P_n\}_{n\in{\mathbb{N}_0}}$ will be the sequence of monic orthogonal polynomials  such that
\[DP_n=P_n
\Lambda_{n} =P_n
\begin{pmatrix}
\lambda _{n} & 0 \\
0 & \mu _{n}
\end{pmatrix}
,\]
with no repetition among $\{\lambda_n,\mu_n\}_{n\in\mathbb N_0}$.
\end{hyp}

To develop the classification mentioned above, we will find the conditions to be satisfied by the matrices $C,U,V$. 
We will prove that the matrices $U$ and $V$ are diagonal matrices (Proposition \ref{UVdiagonales}) and that $C$ can not be a triangular matrix
 (Proposition \ref{C}).  After that, we will prove that $U$ is a scalar multiple of the identity which can be expressed in terms of the entries of the matrices $C$ and $V$ (Proposition \ref{U} and Theorem \ref{UVC}). Finally, we will use an appropriate conjugation to reduce the number of parameters and give the final simplified expression of the operator $D$ (Theorem \ref{CUVAnBn}). 
 %We don' t need to use any explicit expression about the weight $W$.

The first results below are quite simple, while the last ones may be more complicated in the sense that some of their proofs are very technical and require much more complicated computations.

\begin{proposition}\label{UVdiagonales} Every eigenvalue $\Lambda
_{n}
$, for $n\in\mathbb N_0$, is of the form  $$\Lambda _{n}=-n\left(n-1\right) -nU-V.$$ In particular, the matrices $U$ and $V$ are diagonal. 
\end{proposition}
\begin{proof}
From the hypothesis we have that for any $n\in\mathbb N_0$
\begin{equation*}\label{ED}
t(1-t)\frac{d^{2}}{dt^{2}}P_{n}\left( t\right) +\left( C-tU\right) \frac{d}{%
dt}P_{n}\left( t\right) -VP_{n}\left( t\right) =P_{n}\left( t\right) \Lambda
_{n}.
\end{equation*}%
Hence, by looking at the coefficient of degree $n$ we obtain
$$-n(n-1)-nU -V=\Lambda_{n}.$$
This proves the first statement. Considering the case $n=0$ we have that $V$ is diagonal, thus also $U$ is diagonal.
\end{proof}

\begin{proposition} \label{SnAn} The matrices
$S_n$   and $A_{n+1}$
 are  positive definite diagonal matrices, for all $n\in\mathbb N_0$.
\end{proposition}
\begin{proof}
The matrix $S_n$ is given by the square norm of $P_n$, hence it is positive definite for any $n\in\mathbb N_0$.

Since the operator $D$ is symmetric we have $\langle DP_{n},P_{n}\rangle
=\langle P_{n},DP_{n}\rangle $, which implies $\Lambda_n^{\ast }S_{n}=S_{n}\Lambda_n$ and
this, in turn, implies that $S_{n}$ is diagonal since every $\Lambda_n$ is real, diagonal and non-scalar.

Also, from \eqref{AB}, we have that $A_{n}=S_{n-1}^{-1}S_{n}$. Hence, $A_n$ is diagonal and definite positive for every $n\in\mathbb N$.
The proposition is proved.
\end{proof}

\begin{remark}\label{rem}
 Without loss of generality we can assume that the entry $(2,2)$ of $V$, $v_{22}$, is zero since the operator $D-v_{22}I$ and its eigenvalues $\Lambda_{n}-v_{22}I$ still satisfy the conditions given in Hypothesis \ref{h}.
 
  On the other hand, the diagonal entries of $V$ are different to each other  since  $V=-\Lambda_0$.   Hence, we can assume that $$V=\begin{pmatrix}
v & 0 \\
0 & 0%
\end{pmatrix}, \quad \text{ with }    v \ne0.$$
\end{remark}
 
Let us observe that in particular we obtain the following relation between the eigenvalues and the entries of the matrices $U=\left(\begin{smallmatrix}u_1&0\\0&u_2\end{smallmatrix}\right)$ and $V$
\begin{equation}\label{lambdan}
\lambda_n=-n(n-1+u_1)-v ,  \qquad \mu_n= -n(n-1+u_2).
\end{equation}
 This also implies that $\mu_0$ is zero and, from the hypothesis on the eigenvalues, no other $\mu_n$, for $1\le n$, nor  $\lambda_n$, for $0\le n$, is zero.

\

The following lemmas are needed for our explicit computations to obtain a better understanding of the matrices $C=\left(\begin{smallmatrix}c_{11}&c_{12}\\c_{21}&c_{22}\end{smallmatrix}\right)$, $U$ and $V$.

\smallskip

\begin{lemma}\label{33}
If we write $P_n(t)=\sum_{k=0}^nP_n^kt^k$, then we have that
$$
P_n^k=[D,\lambda_n]_k\left(\begin{matrix}
1&0\\0&0
\end{matrix}
\right) + [D,\mu_n]_k\left(\begin{matrix}
0&0\\0&1
\end{matrix}
\right),
$$
where $[D,a]_k$, for $a\in\mathbb R$, is defined recursively by $[D,a]_n=I$ and $$[D,a]_k=(a-\Lambda_k)^{-1}(k+1)(k+C)[D,a]_{k+1}, \quad \text{ for } 0\le k \le n-1.$$
\end{lemma}

\begin{proof}
If, for a fixed $n$, we denote by $$
F\left( t\right)=\sum_{k= 0}^n{t^{k}}F_k$$
%\begin{align*} F\left( t\right)=\sum_{k= 0}^n{t^{k}}F_k&& \text{ and } && G\left( t\right)=\sum_{k= 0}^n{t^{k}}G_k \end{align*} 
the first %and second 
column of $P_n(t)$, then
from the hypothesis we have that $DF=\lambda_nF$. Since $$D(t^k)=t(1-t)k(1-k)t^{k-2}+(C-tU)kt^{k-1}-Vt^k,$$
we have
$$\sum_{k=0}^n 
\left(\left(-k(1-k)-Uk-V\right)t^k+\left(k(1-k)+Ck\right)t^{k-1}\right)F_k=\sum_{k=0}^n \lambda_n t^k F_k.
 $$
Hence, $$
\left(-k(1-k)-Uk-V\right)F_k+\left((k+1)k+C(k+1)\right)F_{k+1}=\lambda_nF_k,
 $$
for $0\le k \le n-1.$
Since $\Lambda_k=-k(1-k)-Uk-V$ we have that 
$$F_k=-(\Lambda_k-\lambda_n)^{-1}(k+1)(k+C)F_{k+1}.$$
Notice that we used that $(\Lambda_k-\lambda_n)$ is non-singular since there is no repetition among $\{\lambda_n,\mu_n\}_{n\in{\mathbb{N}_0}}$.

In a completely analogous way it is proved the corresponding result for the second column of $P_n$, finishing the proof.
\end{proof}
 
\medskip
 
\begin{corollary} \label{Pnn-1} 
If we write $P_n(t)=\sum_{k=0}^nP_n^kt^k$, then we have the following explicit expressions.
\begin{align*}
P_n^{n-1}=&n \begin{pmatrix} 
\frac{c_{11}+n-1}{\lambda_{n}-\lambda_{n-1}} & \frac{c_{12}}{\mu_{n}-\lambda_{n-1}}\\
\frac{c_{21}}{\lambda_{n}-\mu_{n-1}} & \frac{c_{22}+n-1}{\mu_{n}-\mu_{n-1}} 
\end{pmatrix},
\\
 P_{n}^{n-2}=& 
n (n-1)\begin{pmatrix} 
\frac{(c_{11}+n-2)(c_{11}+n-1)}{(\lambda_{n-2}-\lambda_n)(\lambda_{n-1}-\lambda_n)} +\frac{c_{12} c_{21}}{(\lambda_{n-2}-\lambda_n)(\mu_{n-1}-\lambda_n)} & 
\frac{c_{12}(c_{11}+n-2) }{(\lambda_{n-2}-\mu_n)(\lambda_{n-1}-\mu_n)}+\frac{c_{12}(c_{22}+n-1) }{(\lambda_{n-2}-\mu_n)(\mu_{n-1}-\mu_n)} \\
\frac{c_{21}(c_{11}+n-1) }{(\mu_{n-2}-\lambda_n)(\lambda_{n-1}-\lambda_n)}+\frac{c_{21}(c_{22}+n-2) }{(\mu_{n-2}-\lambda_n)(\mu_{n-1}-\lambda_n)} & 
\frac{(c_{22}+n-2)(c_{22}+n-1)}{(\mu_{n-2}-\mu_n)(\mu_{n-1}-\mu_n)} +\frac{c_{12} c_{21}}{(\mu_{n-2}-\mu_n)(\lambda_{n-1}-\mu_n)}
\end{pmatrix}.
\end{align*}
\end{corollary}

\smallskip
\begin{proof}
From Lemma \ref{33} we have 
\begin{align*} 
P_n^{n-1}& =(\lambda_n-\Lambda_{n-1})^{-1} \, n \,  (n-1+C)
\left(\begin{smallmatrix}
1&0\\0&0
\end{smallmatrix}
\right) 
+ (\mu_n-\Lambda_{n-1})^{-1} \, n\,  (n-1+C)
\left(\begin{smallmatrix}
0&0\\0&1
\end{smallmatrix}
\right) \\
&= n \left(\begin{smallmatrix} 
\frac{c_{11}+n-1}{\lambda_{n}-\lambda_{n-1}} & 0\\
\frac{c_{21}}{\lambda_{n}-\mu_{n-1}} & 0
\end{smallmatrix} \right) 
+  
n\left(\begin{smallmatrix} 
0 & \frac{c_{12}}{\mu_{n}-\lambda_{n-1}}\\
 0& \frac{c_{22}+n-1}{\mu_{n}-\mu_{n-1}} 
\end{smallmatrix} \right)
 \intertext{and} 
P_n^{n-2} &= (\lambda_n-\Lambda_{n-2})^{-1}\,(n-1) \,(n-2+C)\,n
 \left(\begin{smallmatrix} 
\frac{c_{11}+n-1}{\lambda_{n}-\lambda_{n-1}} & 0\\
\frac{c_{21}}{\lambda_{n}-\mu_{n-1}} & 0
\end{smallmatrix} \right) \\ 
&+ (\mu_n-\Lambda_{n-2})^{-1}\,(n-1)\, (n-2+C)\,
n \left(\begin{smallmatrix} 
0 & \frac{c_{12}}{\mu_{n}-\lambda_{n-1}}\\
 0& \frac{c_{22}+n-1}{\mu_{n}-\mu_{n-1}} 
\end{smallmatrix} \right).
\end{align*}
After straightforward computations, we obtain the statement of the corollary.
\end{proof}

\

\begin{lemma}\label{34}
If we write $P_n(t)=\sum_{k=0}^nP_n^kt^k$, then the coefficients of the three term recursion relation 
$$
tP_{n}=P_{n+1}+P_{n}B_{n}+P_{n-1}A_{n}, \quad n\in\mathbb N_0,
$$
satisfied by $\{P\left( t\right) \}_{n\in{\mathbb{N}_0}}$ (with the convention $P_{-1}=0$ and $A_0=I$), are explicitly given by 
\begin{align*}
B_n=&P_{n}^{n-1}-P_{n+1}^n,& & \text{ for } n\in\mathbb N_0,
\\
A_n=&P_{n}^{n-2}-P_{n+1}^{n-1}-P_{n}^{n-1}(P_{n}^{n-1}-P_{n+1}^n),& & \text{ for } n\in\mathbb N.
\end{align*}
\end{lemma}
\begin{proof}
If we consider the coefficients of order $n$ and $n-1$ in the three term recurrence relation then we have
\begin{equation*}
B_n=P_{n}^{n-1}-P_{n+1}^n \quad \text{and} \quad A_n=P_{n}^{n-2}-P_{n+1}^{n-1}-P_{n}^{n-1}B_n,
\end{equation*}
respectively. The lemma follows.
\end{proof}

\begin{proposition}\label{C}
If the matrix $C$ in \eqref{OD} is triangular then it is diagonal.
\end{proposition}

\begin{proof}  
Let us assume that $C$ is upper triangular.  The sequence
of matrix-valued orthogonal monic polynomials ${P_{n}}$  satisfy a three-term
recursion relation of the form
\[
tP_{n}=P_{n+1}+P_{n}B_{n}+P_{n-1}A_{n}.
\] %(see \eqref{ttrr} and  \eqref{AB}). 
In particular for $n=0$ we have
\[
t\operatorname{I}=P_{1}+B_{0}
\]
% $S_{0}B_{0}=(S_{0}B_{0})^{\ast }$, with $S_{0}=\Vert P_{0}\Vert ^{2}$.
Thus, we have
$ DP_{1}=(C-tU)-V(t\operatorname{I}-B_{0})=P_{1}\Lambda _{1}, $ % \]%
and 
therefore  $C=-VB_{0}-B_{0}\Lambda _{1}$. By looking at the entries $(1,2)$ and $(2,1)$ we have 
$$\left\{\begin{aligned} c_{12}&=-v(B_0)_{12}-(B_0)_{12} \, \mu_1 ,
\\
0&=-(B_0)_{21} \, \lambda_1. 
\end{aligned}\right.$$
 From here we obtain that $(B_{0})_{12}\neq 0$ because the matrix $C$ is not diagonal, i.e. $c_{12}\neq 0$ and $(B_{0})_{21}=0$, because 
   $\lambda _{1}\neq 0=  \mu_0$. %,   we have that    $(B_{0})_{21}=0$. 
 On the other hand, from \eqref{AB}, we have that  $S_0B_0$ is a hermitian matrix and $S_{0}=\Vert P_{0}\Vert ^{2}$ is a definite positive diagonal matrix (see Proposition \ref{SnAn}). 
Hence we have a contradiction. 
 
%
%Since the operator $D$ is symmetric we have $\langle DP_{0},P_{0}\rangle
%=\langle P_{0},DP_{0}\rangle $, which implies $V^{\ast }S_{0}=S_{0}V$ and
%this, in turn, implies that $S_{0}$ is diagonal (and non-singular) given
%that $V_{0}$ is diagonal and non-scalar. 
% $B_{0}S_{0}$ should be a hermitian matrix.

For $C$ a lower triangular matrix the proof is completely analogous.
\end{proof}

 \
 
Now, thanks to Proposition \ref{C}, we can assume that the entries $(2,1)$ and $(1,2)$ of $C$ are not zero. Otherwise all the matrices $C$, $U$ and $V$ would be diagonal, making $D$ a diagonal differential operator, which reduces to two (scalar) classical Jacobi differential operators.

\begin{proposition}\label{U}
The matrix $U$ is a scalar multiple of the identity, i.e., $U=uI$ with $u\in\mathbb R$.
\end{proposition}

\begin{proof}  
Let us assume that $u_1\neq u_2$.

Combining Lemma \ref{34},  Corollary \ref{Pnn-1} and  \eqref{lambdan}, we can compute explicitly the matrices $A_n$ and $B_n$,  which also must admit a sequence of positive definite matrices  $\{S_n\}_{n\in\mathbb N_0}$ satisfying  \eqref{AB}, i.e., 
\begin{equation*}
A_{n+1}=S_{n}^{-1}S_{n+1,}\quad \text{and } \quad S_{n}B_{n}\text{ is hermitian},\quad \forall n\in\mathbb N_0,
\end{equation*}

From Proposition \ref{SnAn} we have that $A_n$ is a  diagonal matrix, for any $n\in\mathbb N$. %The matrix $A_n$ is given explicitly in term of the coefficients of the monic orthogonal polynomials (see Lemma \ref{34}).  
For $n=1$, by looking the entry $(1,2)$,  %we have  that $A_1=-P_2^0-P_1^0\left( P_1^0-P_2^1\right)$is a diagonal matrix. By using the expicit expressions of the matrices $P_2^0, P_1^0, P_2^1$  given in Corollary  \ref{Pnn-1} and  \eqref{lambdan} 
we obtain that
\begin{equation}\label{v}
 \left( {  u_1}\,c_{{22}}-{  u_2}\,c_{{11}} \right) v+2\,{  u_1}\,{
  u_2}-2\,{  u_1}\,c_{{22}}-2\,{  u_2}\,c_{{11}}=0.
\end{equation}

We start considering the case  $\left( {u_1}\,c_{{22}}-{u_2}\,c_{{11}} \right)=0$.  In this case, from \eqref{v}, we obtain $c_{11}=u_1/2$ and $c_{22}=u_2/2$. 
Now from the condition of that $S_nB_n$ is a hermitian matrix, for $n=0,1,2,3$, we get a linear system of equations for the variables $u_1,u_2, v, c_{21}$ and $c_{12}$. It is matter of straightforward computations to verify that all the solutions lead to a contradiction: in most of these cases, one can check that there exists a non-scalar matrix commuting with every $A_n$ and $B_n$ for $n\in\mathbb N_0$, implying that the matrix weight $W$ reduces (see Theorem \ref{red}); in the other cases we obtain that the diagonal matrix $S_2$ has a null entry, which is a contradiction since every $S_n$ is positive definite.

Now, we can assume that $\left( {  u_1}\,c_{{22}}-{  u_2}\,c_{{11}} \right)\neq0$. Therefore, from \eqref{v}, we can write $v$ in terms of $u_1$,$u_2$,$c_{11}$ and $c_{22}$.
From the condition of $A_2$ being a diagonal matrix (Proposition \ref{SnAn}) we have two conditions:
$$
\left\{
\begin{aligned}
  \left( {  u_1}-{  u_2} \right)  \left( {  u_1}\,c_{{22}}-{  u_2}
\,c_{{11}}+2\,{  u_2}-4\,c_{{22}} \right)  \left( {{  u_1}}^{2}c_{{
22}}-{  u_1}\,{  u_2}\,c_{{11}}-2\,{  u_1}\,{  u_2}+2\,{  u_1}\,c
_{{22}}+2\,{  u_2}\,c_{{11}} \right) =0,\\
 \left( {  u_1}-{  u_2} \right)  \left( {  u_1}\,c_{{22}}-{  u_2}
\,c_{{11}}-2\,{  u_1}+4\,c_{{11}} \right)  \left( {  u_1}\,{  u_2}
\,c_{{22}}-{{  u_2}}^{2}c_{{11}}+2\,{  u_1}\,{  u_2}-2\,{  u_1}\,c
_{{22}}-2\,{  u_2}\,c_{{11}} \right)=0.
\end{aligned}
\right.
$$
Since  $u_1\neq u_2$, we have
\begin{equation}\label{aux}\left\{
\begin{split}
\overbrace{  \left( {  u_1}\,c_{{22}}-{  u_2}
\,c_{{11}}+2\,{  u_2}-4\,c_{{22}} \right) }^\text{\raisebox{.5pt}{\textcircled{\raisebox{-.9pt} {1}}} 
}
 \overbrace{\left( {{  u_1}}^{2}c_{{
22}}-{  u_1}\,{  u_2}\,c_{{11}}-2\,{  u_1}\,{  u_2}+2\,{  u_1}\,c
_{{22}}+2\,{  u_2}\,c_{{11}} \right)}^\text{\raisebox{.5pt}{\textcircled{\raisebox{-.9pt} {2}}} }
 =0, 
\\
\underbrace{ \left( {  u_1}\,c_{{22}}-{  u_2}
\,c_{{11}}-2\,{  u_1}+4\,c_{{11}} \right)}_\text{\raisebox{.5pt}{\textcircled{\raisebox{-.9pt} {3}}} }
 \underbrace{ \left( {  u_1}\,{  u_2}
\,c_{{22}}-{{  u_2}}^{2}c_{{11}}+2\,{  u_1}\,{  u_2}-2\,{  u_1}\,c
_{{22}}-2\,{  u_2}\,c_{{11}} \right)}_\text{\raisebox{.5pt}{\textcircled{\raisebox{-.9pt} {4}}} }
=0.
\end{split}\right.
\end{equation}

 By considering the possibilities given by \eqref{aux}, we can consider four cases which, after the corresponding elementary computations, lead to a contradiction: 
\begin{enumerate} 
\item $\raisebox{.5pt}{\textcircled{\raisebox{-.9pt} {1}}}=0=\raisebox{.5pt}{\textcircled{\raisebox{-.9pt} {3}}}$: \newline
From these conditions, we obtain $u_1=2c_{11}$ and $u_2=2c_{22}$, having then a contradiction since $\left( {  u_1}\,c_{{22}}-{  u_2}\,c_{{11}} \right)\neq0$.
\item $\raisebox{.5pt}{\textcircled{\raisebox{-.9pt} {1}}}=0=\raisebox{.5pt}{\textcircled{\raisebox{-.9pt} {4}}}$:\newline
In this case, we obtain
 $u_1=\frac{2(c_{11}-2-2c_{22})}{c_{11}-c_{22}-2}$ and $u_2=-\frac{2c_{22}}{c_{11}-c_{22}-2}$.  Hence, we obtain $\lambda_0=\mu_1$, which contradicts our hypothesis of no repetition among $\{\lambda_n,\mu_n\}_{n\in\mathbb N_0}$.
\item  $\raisebox{.5pt}{\textcircled{\raisebox{-.9pt} {2}}}=0=\raisebox{.5pt}{\textcircled{\raisebox{-.9pt} {3}}}$:\newline
This leads to $u_1=\frac{2c_{11}}{c_{11}-c_{22}+2} $ and $u_2= \frac{2(2c_{11}+2-c_{22})}{c_{11}-c_{22}+2}$. From this one obtains $\lambda_1=0=\mu_0$, which contradicts our hypothesis of no repetition among $\{\lambda_n,\mu_n\}_{n\in\mathbb N_0}$.
\item  $\raisebox{.5pt}{\textcircled{\raisebox{-.9pt} {2}}}=0=\raisebox{.5pt}{\textcircled{\raisebox{-.9pt} {4}}}$:\newline
These conditions imply $u_1=\frac{2(c_{11}-c_{22})}{c_{11}+c_{22}+2}$ and $u_2=-\frac{2(c_{11}-c_{22})}{c_{11}+c_{22}+2}$, which in turn imply $\lambda_0=\mu_1$, which contradicts our hypothesis of no repetition among $\{\lambda_n,\mu_n\}_{n\in\mathbb N_0}$.
\end{enumerate}
All these cases give a  contradiction as consequence of the assumption $u_1 \neq u_2$.
Therefore  $u_1=u_2$ and we concluded the proof. 
\end{proof}

\medskip

\begin{theorem}
We have that 
\begin{equation}
\label{UVC}
U=\big(c_{11}  +  c_{22}+  \tfrac{1}{2}v(c_{11}  -   c_{22}) \big) I, \quad
V=
\begin{pmatrix}
v & 0 \\
0 & 0%
\end{pmatrix}, \quad C=
\begin{pmatrix}
c_{11} & -\frac{(c_{11}-c_{22}-2)(c_{11}-c_{22}+2)}{4c_{21}} \\
c_{21} & c_{22}
\end{pmatrix},
\end{equation}
with $v,c_{11},c_{22},c_{21}\in\mathbb{R}$ and $v,c_{21}\neq 0$.
\end{theorem}

\begin{proof}
From Propositions \ref{C} and \ref{U}, we know that 
\begin{align*}
U=u I, &&
V=
\begin{pmatrix}
v & 0 \\
0 & 0
\end{pmatrix}, &&
 C=
\begin{pmatrix}
c_{11} & c_{12} \\
c_{21} & c_{22}
\end{pmatrix},
\end{align*}
with $v,c_{11},c_{22},c_{21},c_{12},u\in\mathbb{R}$ and $v,c_{21},c_{12}\neq 0$.

For any $n\in\mathbb N$, let us focus on  the $t^{n-2}$ coefficient in the three-term recurrence relation \eqref{ttrr}:
$$tP_{n}=P_{n+1}+P_{n}B_{n}+P_{n-1}A_{n}.$$
If  write $P_n=\sum_{k=0}^n P_n^k t^k$, we have
\begin{equation}\label{eq}
P_n^{n-3}-P_{n+1}^{n-2}-P_{n}^{n-2}B_n-P_{n-1}^{n-2}A_n=0.
\end{equation}
From Lemmas \ref{33} and \ref{34}, one can compute explicitly  the entries of the matrix equation \eqref{eq}.
The entry $(1,1)$ is given by
\[
\frac{%1}{3}\frac{%
c_{12}c_{21}\, v\,n(n-1)\,  (4n+3 u-8)\, (vc_{11}-vc_{22}+2c_{11}+2c_{22}-2u)}{%
3(4n+v+2u-6)(2n-v+u-4)(2n+u-4)(4n-v+2u-6)(2n+v+u-2)(2n+u-2)}=0.
\]%
Then , since $c_{12},c_{21},v\neq 0$, we have that this equation holds for every $n\in\mathbb N$ only if  \[
u=\tfrac{1}{2}v(c_{11}-c_{22})+c_{11}+c_{22}.
\]%
 
Now, by looking at the entry $(2,1)$ of the matrix equation \eqref{eq}, we have
\begin{eqnarray*}
&&\frac{(3vc_{11}-3vc_{22}+8n+2v+6c_{11}+6c_{22}-16)}{%
(3vc_{11}-3vc_{22}+12n+2v+6c_{11}+6c_{22}-24)(3vc_{11}-3vc_{22}+12n+2v+6c_{11}+6c_{22}-12)
}
\times\\
&&\frac{(c_{11}^{2}+c_{22}^{2}-2c_{11}c_{22}-4+4c_{12}c_{21})}{%
(vc_{11}-vc_{22}+4n+2v+2c_{11}+2c_{22}-8)(vc_{11}-vc_{22}+4n+2c_{11}+2c_{22}-6)%
}
\times
 \\
&&\frac{4c_{21}\,v\,n(n-1)}{(vc_{11}-vc_{22}+4n+2v+2c_{11}+2c_{22}-4)}=0,
\end{eqnarray*}%
which holds for every $n\in\mathbb N$ only if
\[
-4c_{21}c_{12}=(c_{11}-c_{22}-2)(c_{11}-c_{22}+2).
\]%
Hence, we have that $C,U,V$ are given by \eqref{UVC}.
\end{proof}

Finally, we can make a change of parameters to express our operator (or the matrices $C,U,V$) in a simpler form. Also, we can consider now only three parameters instead of the four $v,c_{11},c_{22},c_{21}$, as stated in the following theorem.

\begin{theorem}\label{CUVAnBn}
The pair $(W,D)$ is, up to a scalar translation of $D$ (see Remark \ref{rem}), always equivalent to a pair in which the matrices $C,U,V$ are of the form
\[
C=%
\begin{pmatrix}
\alpha +2-\frac{\alpha -\beta }{v} & \frac{v+\alpha -\beta }{v} \\
\frac{v-\alpha +\beta }{v} & \alpha +2+\frac{\alpha -\beta }{v}%
\end{pmatrix}%
,\qquad U=\left( \alpha +\beta +4\right) \mathrm{\ I},\qquad V=%
\begin{pmatrix}
v & 0 \\
0 & 0%
\end{pmatrix}%
,
\]
with $\alpha, \beta, v\in\mathbb R$.
Furthermore, the matrices involved in the three-term recurrence relation \eqref{ttrr} are now given by

\begin{align*}
A_n&={\frac {n   \left( \alpha+n+1 \right) \left( \beta+n+1 \right) \left( \alpha+\beta+n+2
 \right)  }{
 \left( \alpha+\beta+2\,n+v+2 \right)  \left( \alpha+\beta+2\,n+3 \right)  \left( \alpha+\beta+2
\,n+2 \right) ^{2} \left( \alpha+\beta+2\,n-v+2 \right)  \left( \alpha+\beta+2\,n+1
 \right) }}\times 
\\ & \qquad 
\begin{pmatrix}
 \left( v+2\,n+\beta+\alpha \right)  \left( \alpha+\beta+2\,n-v+4 \right) &0 \\
0 &  \left( -v+2\,n+\beta+\alpha \right)  \left( \alpha+\beta+2\,n+v+4 \right)%
\end{pmatrix}
\end{align*}

\noindent and the entries of $B_n$ are
\begin{align*}
\left( B_{n}\right) _{11} &=-n\frac{(\alpha+n+1)v+\beta-\alpha}{
(\alpha+\beta+2n+2)v}+(n+1)\frac{(\alpha+n+2)v+\beta-\alpha}{
(\alpha+\beta+2n+4)v}, \\
\left( B_{n}\right) _{12} &={\frac { \left( v+\alpha-\beta  \right)  \left( \alpha+\beta -v+2 \right) }{v \left(\alpha+\beta +2
\,n-v+2 \right)  \left( \alpha+\beta +2\,n-v+4 \right) }},\\
\left( B_{n}\right) _{21} &={\frac { \left( v-\alpha+\beta \right)  \left( \alpha+\beta +v+2 \right) }{v \left( \alpha+\beta +2
\,n+v+2 \right)  \left( \alpha+\beta +2\,n+v+4 \right) }}
, \\
\left( B_{n}\right) _{22} &=-n\frac{(\alpha+n+1)v-\beta+\alpha}{%
(\alpha+\beta+2n+2)v}+(n+1)\frac{(\alpha+n+2)v-\beta+\alpha}{%
(\alpha+\beta+2n+4)v}
.
\end{align*}
\end{theorem}

\begin{proof}
Let us consider the expressions in \eqref{UVC} and let us define the matrix 
$$M=
\begin{pmatrix}
\frac{c_{11}-c_{22}+2}{2c_{21}}&0\\
0&1
\end{pmatrix}.
$$
Observe that $c_{11}-c_{22}+2\neq0$ and $c_{21}\neq0$. Otherwise, from \eqref{UVC}, we would have that $C$ is triangular and then diagonal (Lemma \ref{C}), therefore the entire operator would be diagonal.

 If we conjugate the pair $(W,D)$ by a matrix $M$ we obtain the pair $(M^*WM,M^{-1}DM)$. 
 Hence 
we have that 
 \begin{align*}
M^{-1}UM&= U=\left( \frac{1}{2}v(c_{11}-c_{22})+c_{11}+c_{22}\right) I, 
&
M^{-1}VM= V=
\left(
\begin{matrix}
v & 0 \\
0 & 0%
\end{matrix}\right)
,
\\
M^{-1}CM&=
\left(
\begin{matrix}
c_{11} & -\frac{1}{2}\left( c_{11}-c_{22}-2\right) \\
\frac{1}{2}\left( c_{11}-c_{22}+2\right) & c_{22}%
\end{matrix}\right).
&
\end{align*}

Applying the change of parameters
$$
\alpha =\tfrac{1}{2}(c_{11}+c_{22})-2 \quad \text{ and } \quad
\beta =\tfrac{1}{2}(c_{11}-c_{22})v+\tfrac{1}{2}(c_{11}+c_{22})-2
$$
we obtain the expression of $C$ and $U$. The eigenvalues are given by 
$$ \lambda_n= -n(n+\alpha+\beta+3)-v, \quad \text{ and } \quad \mu_n= -n(n+\alpha+\beta+3).$$

 By combining carefully Lemma \ref{34}  and Corollary \ref{Pnn-1} we obtain the expressions of  $B_n$ and $A_n$ given in the statement.  
\end{proof}

\begin{theorem}\label{tabv}
The parameters $\alpha$, $\beta$ and $v$
satisfy the following inequalities 
\begin{equation}\label{abv}
|\alpha-\beta|< |v| < \alpha+\beta +2.
\end{equation}
\end{theorem}
\begin{proof}
Let us focus on the fact that $S_n$ is a positive definite diagonal  matrix  and $S_nB_n$ is an Hermitian matrix for any $n\in\mathbb N_0$ (see \eqref{AB}).  In particular we have that,  for any $ n\geq 0$, $(B_n)_{12}$ and $(B_n)_{21}$ are numbers with both positive or both negative.
From Theorem \ref{CUVAnBn} we  have that 
$$(B_n)_{12}(B_n)_{21}=\frac{(v+\alpha-\beta)(v-\alpha+\beta)(\alpha+\beta+2-v)(\alpha+\beta+2+v)}
{(\alpha+\beta+2n+2-v)(\alpha+\beta+2n+2+v)(\alpha+\beta+2n+4-v)(\alpha+\beta+2n+4+v)}$$
for any $n\in\mathbb N_0$.
This implies that 
\begin{equation}\label{d1}
(v^2-(\alpha-\beta)^2)((\alpha+\beta+2)^2-v^2) > 0
\end{equation}
and 
\begin{equation}
\\\label{d2}
(\alpha+\beta+2n-v)(\alpha+\beta+2n+v)(\alpha+\beta+2n+2-v)(\alpha+\beta+2n+2+v) > 0,
\end{equation}
for all $n\in \mathbb N$.

If we change $v$ by $-v$ in \eqref{d1} and \eqref{d2} we obtain exactly the same conditions, then it is enough  to consider the situation  $v>0$.  The condition  \eqref{d1} implies that 
$$|\alpha-\beta|<|v|<|\alpha+\beta+2| \quad \text{ or }\quad |\alpha+\beta+2|<|v|<|\alpha-\beta|.$$

If $|\alpha+\beta+2|<|v|$ then $-v<\alpha+\beta+2<v$, hence  there exists $k\in \mathbb N$ such that 
$$\alpha+\beta+2k <v<\alpha+\beta+2k+2$$ 
and $-v<\alpha+\beta+2k$. Thus \eqref{d2}, for $n=k$, gives a contradiction.   

Therefore we have $$|\alpha-\beta|<v<|\alpha+\beta+2|.$$ 
Now, we prove that $\alpha+\beta+2>0$. If $\alpha+\beta+2<0$ then we have $\alpha+\beta+2+v<0$.
  Thus there exists $k\in\mathbb N$ such that 
\begin{equation}\label{daux}
\alpha+\beta+2k+v <0<\alpha+\beta+2k+2+v.
\end{equation}
From \eqref{d2}, with $n=k$, we get that $0<\alpha+\beta+2k-v+2$. In particular this implies 
\begin{equation}\label{vv}0<v<1, \qquad \alpha+k+1>0, \qquad  \beta+k+1>0.
\end{equation}

On the other hand we have that $A_n$ is a diagonal definite positive matrix, then the entry  $(A_n)_{22}$ is a positive number.  
From Theorem \ref{CUVAnBn} we obtain that, for all $n\in \mathbb N$,
\begin{equation} \label{d3}
\frac { \left( \alpha+n+1 \right)  \left( \beta+n+1 \right)  \left( \alpha+\beta+n+2 \right)  \left( \alpha+\beta+2n-v \right)  \left( \alpha+\beta+2n+v+4 \right) }{
 \left(\alpha+\beta+2n+3 \right) \left( \alpha+\beta+2n+1 \right)  \left( \alpha+\beta+2n-v+2 \right) \left( \alpha+\beta+2n+v+2 \right)   }>0. 
 \end{equation}
For $n=k$, by using the \eqref{vv},  we obtain  $\frac{ \left( \alpha+\beta+k+2 \right)}{ \left( \alpha+\beta+2k+1 \right) }<0$.  If $k=1$, this gives a contradiction. If $k\geq 2$, this implies that $ \alpha+\beta+k+2  <0<  \alpha+\beta+2k+1 .$
Then, from \eqref{d3} with $n=k-1$ we obtain 
\begin{equation*} 
\frac { \left( \alpha+k \right)   \left( \beta+k \right)  \left( \alpha+\beta+2k-2-v \right)  }{
 \left( \alpha+\beta+2k-1 \right)  \left( \alpha+\beta+2k-v \right) \left( \alpha+\beta+2k+v \right)   }<0. 
 \end{equation*}
From, \eqref{daux}, we then obtain 
$${(\alpha+k)(\beta+k)}<0.$$ 

  In the case   $\alpha +k>0$ and $\beta+k<0$ we use that $\alpha-\beta<v$, to obtain $\alpha
+\beta+2k+v>0$,  which is a contradiction.  In the case   $\alpha +k<0$ and $\beta+k>0$ we use that $-\alpha+\beta<v$, to obtain again the contradiction $\alpha
+\beta+2k+v>0$. 
Therefore we have  $\alpha+\beta+2>0$ and we have proved the statement in the theorem.
\end{proof}

Notice that if we exclude $v$ from the condition \eqref{abv} we have $|\alpha-\beta|<\alpha+\beta+2$, which is equivalent to $-1<\alpha$ together with $-1<\beta$. As the reader may have noticed, these conditions on $\alpha$ and $\beta$ are the same to those for the parameters in the scalar Jacobi orthogonal polynomials.

There is no further reduction on the parameters $\alpha$, $\beta$ or $v$ to be made, as we show in the next  section. This is done in the most explicit way: by giving all the pairs $(D,W)$.

\section{The Three Parameter Family}\label{Existence}

In this section, we give explicitly a three-parameters family of $2\times 2$-matrix irreducible weights $\{W\}$, on the interval $(0,1)$, such that for each $W$ there exists a hypergeometric symmetric operator 
\[
D= t(1-t) \frac{d^2}{dt^2}+ (C-tU) \frac {d}{dt}-V
\] 
of order two that diagonalizes on the basis of monic orthogonal polynomials $\{P_n\}_{n\in\mathbb N_0}$, i.e.,
$$
DP_n=P_n \Lambda_n, \quad \text{ for all } n\in{\mathbb{N}_0},
$$
 with 
 $\{\Lambda_n\}_{n\in{\mathbb{N}_0}}$ diagonal matrices.

\begin{proposition}\label{W}
Let $\alpha, \beta ,v\in\mathbb R$ such that $|\alpha-\beta|< |v| < \alpha+\beta +2$. For $%
t\in (0,1)$, let
\begin{equation}\label{wabv}
W_{\alpha,\beta,v}(t) =t^{\alpha }\left( 1-t\right) ^{\beta }\,\left(
W_{2}t^{2}+W_{1}t+W_{0}\right),
\end{equation}
with
\begin{align*}
W_{2} &=
\begin{pmatrix}
\frac{v( v+2+\alpha +\beta)}{v+\alpha-\beta} & 0 \\
0 & \frac{v(- v+2+\alpha +\beta)}{v-\alpha+\beta}%
\end{pmatrix},
&
W_{1} &=
\begin{pmatrix}
-\left( v+\alpha +\beta +2\right) & ( \alpha +\beta+2 ) \\
( \alpha +\beta+2 ) & -\left(- v+\alpha +\beta +2\right)%
\end{pmatrix},
\\
W_{0} &= (\alpha+1)
\begin{pmatrix}
1 & -1 \\
-1 & 1%
\end{pmatrix}.
\end{align*}
Thus, $W_{\alpha,\beta,v}$ is a matrix weight and the differential operator given by $$D= t(1-t) \frac{%
d^2}{dt^2}+ (C-tU) \frac {d}{dt}-V,$$ with
\[
C =
\begin{pmatrix}
\alpha +2-\frac{\alpha -\beta }{v} & \frac{v+\alpha-\beta }{v} \\
\frac {v-\alpha+\beta}v & \alpha +2+\frac{\alpha -\beta }{v}%
\end{pmatrix}%
,\qquad U =\left( \alpha +\beta +4\right) \mathrm{\ I}, \qquad V=%
\begin{pmatrix}
v & 0 \\
0 & 0%
\end{pmatrix}%
,
\]
  is symmetric with respect to    $W_{\alpha,\beta,v}$.
\end{proposition}

\begin{proof}
To ease the notation in this proof, let us denote $W=W_{\alpha,\beta,v}$. The matrix-function $W(t)$ is symmetric for any $t\in(0,1)$ and, since $\alpha, \beta> -1$, all the moments are finite. Hence, to prove that it is a weight we only need to prove that it is a definite matrix for every $t\in (0,1)$. The determinant of $W$ is given by 
$$t^{2(\alpha +1)}\left( 1-t\right) ^{2(\beta+1)} v^2{\frac {  \left( -v+\alpha+\beta+2
 \right)  \left( v+\alpha+\beta+2 \right) }{ \left( v+\alpha-\beta \right)  \left( v-(\alpha-
\beta) \right) }},
$$
which is clearly positive since $|\alpha-\beta|< |v| < \alpha+\beta +2$. Therefore, to prove that $W(t)$ is positive-definite it suffices to prove that the entry $(1,1)$ is positive. This entry $(1,1)$ is given by the product of the positive number $t^{\alpha }\left( 1-t\right)^\beta$ and the polynomial on $t$
$$\alpha+1-t \left(  v+\alpha+\beta+2 \right) +{t}^{2} \; {\frac {v \left( v+\alpha+\beta+2
 \right) }{v+\alpha-\beta}}.
$$
It is very easy to see that the minimum of this polynomial is 
$$\frac{1}{4}\,{\frac { \left( -v+\alpha+\beta+2 \right)  \left( v-(\alpha-\beta) \right) }{v}},
$$
is clearly positive since $|\alpha-\beta|< |v| < \alpha+\beta +2$. Hence, $W$ is positive-definite and therefore a matrix-weight.

It is matter of a careful integration by parts to see that the condition of
symmetry for a differential operator of order two
%, given in Definition \ref{symm},
is equivalent to a set of three differential equations involving the weight $%
W$ and the coefficients of the differential operator $D$ and the boundary
conditions, see \cite{GPT03,DG04}. 
In this case, the differential operator $D$ is symmetric with respect to $W$ if and only if  %it satisfy %
\begin{equation}  \label{SDE1}
\left\{\; \begin{split}
2\left(  t (1-t) W\right) ^{\prime }& =W(C-tU)+(C-tU)^{\ast }W, \\
\left( t(1-t)\right) ^{\prime \prime }-\left( W(C-tU)\right) ^{\prime
}& =WV-V^{\ast }W,
\end{split}\right.
\end{equation}
together  with the boundary conditions
\begin{equation}\label{SDE2}
\underset{t\rightarrow 0,1}{\lim }t(1-t) W(t)  =0\quad \text{ and } \quad  \underset{%
t\rightarrow 0,1}{\lim }\big(W(t)(C-tU)-(C-tU)^{\ast }W(t)\big)=0.
\end{equation}

 Let us observe that,
\begin{align*}
 (W(t)t(1-t))^{\prime }=  t^\alpha (1-t)^\beta \Big (&  - (\alpha+\beta+4) W_2 t^3+( (\alpha+3)W_2-(\alpha+\beta+3)W_1)t^2 \\
 & + ( (\alpha+2)W_1-(\alpha+\beta+2)W_0)t+ (\alpha+1)W_0\Big),
 \end{align*}
and 
\begin{align*}
  W(t)& (C-tU)+  (C-tU)^{\ast }W(t)= t^\alpha (1-t)^\beta \Big (  - 2(\alpha+\beta+4) W_2 t^3\\
  +& ( W_2C+C^*W_2-2(\alpha+\beta+4)W_1)t^2  + ( W_1C+C^*W_1-2(\alpha+\beta+4)W_0)t+W_0C+C^*W_0 \Big),
\end{align*}
 Now, it is easy to verify that the first equation of \eqref{SDE1} is satisfied.

The second equation in \eqref{SDE1} is equivalent (by using the first one) to 
\begin{equation}\label{3eqsymm}
  \big(W(C-tU)-(C-tU)^*W\big)^\prime= 2(V^*W-WV).
\end{equation}

We have   $W_2(C-tU)-(C-tU)^* W_2=  2v \left(\begin{smallmatrix}
  0& 1\\ -1& 0
\end{smallmatrix}\right)
=-(W_1(C-tU)-(C-tU)^* W_1)$ and $W_0(C-tU)-(C-tU)^*W_0=0$.
Thus  $$W(C-tU)-(C-tU)^*W= -2vt^{\alpha+1} (1-t)^{\beta+1} \left(\begin{smallmatrix}
  0& 1\\ -1& 0
\end{smallmatrix}\right).$$
On the other hand $$V^*W-WV= -vt^{\alpha} (1-t)^{\beta} (\alpha+1-t(\alpha+\beta+2))\left(\begin{smallmatrix}
  0& 1\\ -1& 0
\end{smallmatrix}\right),$$
and \eqref{3eqsymm} holds.

Finally, the boundary conditions $\lim_{t\to \pm 1} t(1-t)W(t)=0$ and $\lim_{t\to \pm 1} W(t)F_1(t)-F_1^*(t)W(t)=0$ hold because $\alpha, \beta>-1$.
\end{proof}

\begin{proposition}\label{wirr}
The weight $W_{\alpha,\beta,v}$ is irreducible.
\end{proposition}
\begin{proof}
Let us assume that there exists a nonsingular matrix
$M=\left(\begin{smallmatrix}
m_{11}  & m_{12} \\
m_{21} & m_{22}%
\end{smallmatrix}\right).
$
such that
\[
M^{*}W_{\alpha,\beta,v}(t)M=%
\begin{pmatrix}
w_1(t) & 0 \\
0 & w_2(t)%
\end{pmatrix}.
\]
The entry $(1,2)$ of $M^{*}W_{\alpha,\beta,v}(t)M$ is the following polynomial on $t$
\begin{multline*}
\left(m_{11}m_{12}\frac{v( v+2+\alpha +\beta)}{v+\alpha-\beta}+m_{22}m_{21}\frac{v(-v+2+\alpha +\beta)}{v-\alpha+\beta}\right)t^{2} \\
%((\alpha+1)^2-(\beta+1)^2 -2\beta
%v-v^2)(m_{12}m_{11}-m_{21}m_{22})-2v(m_{12}m_{11}+m_{21}m_{22}) t^2 \\
+\big ((\alpha+\beta+2)(m_{12}-m_{22})(m_{21}-m_{11}) +
(m_{22}m_{21}-m_{11}m_{12})\, v\big ) t  + (m_{12}-m_{22})(m_{11}-m_{21})(\alpha+1).
\end{multline*}
Then, by looking at the constant coefficient, since $\alpha>-1$ we have that
\begin{equation}\label{irr}
(m_{12}-m_{22})(m_{11}-m_{21})=0.
\end{equation}

Now by looking  at the linear coefficient we have
that
$
%(\alpha+\beta+2)(m_{12}-m_{22})(m_{21}-m_{11}) +
(m_{22}m_{21}-m_{11}m_{12})v=0, 
$
which implies 
\begin{equation}\label{irr2}
m_{22}m_{21}-m_{11}m_{12}=0,
\end{equation}
 since $v\neq 0$. Now, combining \eqref{irr} and \eqref{irr2} we have that $\det(M)=0$
which is  a contradiction.
\end{proof}

% \bigskip \bigskip
%
% We want to construct a sequence $\{P_n\}_{n\in\mathbb N_0}$ of matrix-valued
% orthogonal polynomials with respect to de weight function $W$, with degree
% of $P_{n}$ equal to $n$, with no singular leading coefficient and which
% satisfies $DP_{n}=P_{n}\Lambda_{n}$ where $\Lambda_{n}$ is a real diagonal
% matrix.

%\bigskip \bigskip

Combining the previous results in this paper we can now state the following theorem.
\begin{theorem}\label{main}
An operator satisfies Hypothesis \ref{h} if and only if it is equivalent to an operator of the form 
$$D_{\alpha,\beta,v,v_2}= t(1-t) \frac{%
d^2}{dt^2}+ (C-tU) \frac {d}{dt}-V,$$ with
\[
C =
\begin{pmatrix}
\alpha +2-\frac{\alpha -\beta }{v} & \frac{v+\alpha-\beta }{v} \\
\frac {v-\alpha+\beta}v & \alpha +2+\frac{\alpha -\beta }{v}%
\end{pmatrix}%
,\qquad U =\left( \alpha +\beta +4\right) \mathrm{\ I}, \qquad V=%
\begin{pmatrix}
v & 0 \\
0 & 0%
\end{pmatrix}+v_2\mathrm{\ I}
,
\]
and $\alpha, \beta ,v,v_2\in\mathbb R$ such that $|\alpha-\beta|< |v| < \alpha+\beta +2$.
Every $D_{\alpha,\beta,v,v_2}$ is symmetric with respect to the irreducible weight $W_{\alpha,\beta,v}$ given in Proposition \ref{wabv}.

Furthermore, let $\{P_{n}\}_{n\in\mathbb N_0}$ be the sequence of matrix-valued orthogonal monic
polynomials associated to $W_{\alpha,\beta,v}$. Then,  $P_{n}$
is an eigenfunction of the differential operator $D$ with diagonal eigenvalue
\[
\Lambda _{n}=%
\begin{pmatrix}
\lambda _{n} & 0 \\
0 & \mu _{n}%
\end{pmatrix},
\]%
where,
\begin{align*}
\lambda _{n}& =-n(n-1)-n\left( \alpha +\beta +4\right) -v \quad  \text{ and } \quad 
\mu _{n}  =-n(n-1)-n\left( \alpha +\beta +4\right).
\end{align*}
Also, if two operators $D_{\alpha,\beta,v,0}$ and $D_{\alpha',\beta',v',0}$ are equivalent, then  
 $(\alpha,\beta,v,0)=(\alpha',\beta',v',0)$.
\end{theorem}
\begin {proof}

Since $D$ is symmetric, by
\cite[Prop. 2.10 and 2.7]{GT07}, we know that $P_n$ is eigenfunction of $D$, 
with eigenvalues $\Lambda_n= n(n-1)-nU-V$   for every $n\in\mathbb N_0$, obtaining the expression given in the theorem. It is readily seen that there is no repetition among $\{\lambda_n,\mu_n\}_{n\in\mathbb N_0}$ given that $|\alpha-\beta|< |v| < \alpha+\beta +2$.

Hence, we only need to prove the last statement, since the first ones follow from Theorem \ref{CUVAnBn}, Theorem \ref{tabv}, Proposition \ref{W} and Proposition \ref{wirr}. 

Let us assume that there exists a non-singular matrix $M$ such that 
$$
D_{\alpha,\beta,v,v_2}=M^{-1}D_{\alpha',\beta',v',v'_2}M.
$$
Then, by looking at the coefficient of order zero we have
$
M\left(\begin{smallmatrix}
v & 0 \\
0 & 0%
\end{smallmatrix}\right)=\left(\begin{smallmatrix}
v' & 0 \\
0 & 0%
\end{smallmatrix}\right)M
$. Hence $v=v'$ and the matrix $M$ is diagonal. 

By looking at the coefficient of order one it is easy to obtain that $\alpha=\alpha'$ and $\beta=\beta'$.
\end{proof}

\subsection{Some comments on the vector-valued hypergeometric function}

\

The matrix-valued hypergeometric equation 
 $$ t(1-t) \frac{%
d^2 F}{dt^2}+ (C-tU) \frac {dF}{dt}-VF=0,$$
 has been studied in \cite{T03}. If the eigenvalues of $C$ are not in $-\mathbb N_0$, the analytic solutions, around $t=0$, 
  are determined by $F_0=F(0)$ and given explicitly by 
 the function %$\,_{2}H_{1}$ is given by

 $$
F(t)=\,_{2}H_{1}\left(
{\genfrac{}{}{0pt}{}{U,V}{C}}
;t\right) F_0 =\sum_{k= 0}^{n} \frac{t^{k}}{k!}\left[ C,U,V\right] _{k}F_0, \qquad |t|<1.$$
where the symbol $[ C,U,V]_{n}$ is inductively defined by $\left[ C,U,V\right] _{0} =I$ and

\begin{eqnarray*}
\left[ C,U,V\right] _{n+1} =\left( C+n\right) ^{-1}\left( n^{2}+n\left(
U-1\right) +V\right) \left[ C,U,V\right] _{n}\quad \text{ for } \quad n\in\mathbb N_0.
\end{eqnarray*}
Observe that in our case, from the expression of the matrix $C$ in Proposition \ref{W}, it is easy to see that the eigenvalues of $C$ are not in$-\mathbb N_0$ since they are given by $\alpha+1$ and $\alpha+3$, with $-1<\alpha$.

\

\subsection{An example coming from group representation theory: The Gegenbauer matrix-weight.}

In \cite{PZ16} the authors  study $2\times  2$ matrix-valued orthogonal polynomials
associated with spherical functions in the $q$-dimensional sphere $S^q$
( originally $q$ was a natural
number, but this parameter was later extended to any real positive number). The weight matrix,
depending on parameters $ 0 < p < q$, is given by
\begin{equation*}
  W_{p,q}(x)= (1-x^2)^{\tfrac q 2 -1} \begin{pmatrix}
  p\,x^2+q-p & -q x\\ -q x & (q-p)x^2+p
\end{pmatrix},\qquad x\in [-1,1].
\end{equation*} 
 The monic orthogonal polynomials $\{R_n\}$, associated to this weight matrix satisfy
 % For each $w\in \NN_0$, the matrix  polynomial $R_w$ satisfies
$DR_n= R_n \Lambda_n $, where
%    $is an eigenfunction of the differential operator
$$D= (1-x^2) \frac{d^2}{dx} -\Big( (q+2)x+2\left(\begin{smallmatrix}
  0&1\\1&0\end{smallmatrix}\right) \Big) \frac{d}{dx} -\left(\begin{smallmatrix}
    p&0\\0&q-p  \end{smallmatrix}\right), $$
and the eigenvalue   is given by the following diagonal matrix
$$\Lambda_n(D)= \begin{pmatrix}
  -n(n+q+1)-p & 0\\ 0& -n(n+q+1)-q+p
\end{pmatrix}.$$

This example is a particular case of the family presented in this paper: after the change of variable $x=1-2t$, we obtain that 
$D=D_{\alpha, \beta, v, v_2}$ and $W=2^{q-1} W_{\alpha, \beta, v}$  with
$$ \alpha=\beta= \frac q2-1, \quad v=2p-q, \quad v_2=q-p.$$

\bibliography{ref18}

\bibliographystyle{alpha}

\end{document}